\documentclass[a4paper,12pt]{article}
\usepackage[top=2.5cm,bottom=2.5cm,left=2.5cm,right=2.5cm]{geometry}
\usepackage{cite, amsmath, amssymb}
\usepackage{amssymb}
\usepackage{graphicx}
\newenvironment{proof}{{\noindent\textbf{\text \it {Proof.}}}}{\hfill $\blacksquare$\par}
\usepackage{latexsym}
\usepackage{graphicx,booktabs,multirow}
\usepackage{tikz}
\usetikzlibrary{decorations.pathreplacing}
\usetikzlibrary{intersections}
\usepackage{enumerate}
\usepackage{graphicx,booktabs,multirow}
\usepackage{appendix}
\newtheorem{theorem}{Theorem}[section]
\newtheorem{definition}[theorem]{Definition}
\newtheorem{lemma}[theorem]{Lemma}

\newtheorem{corollary}[theorem]{\rm\bfseries Corollary}

\usepackage[section]{placeins}
\def\NAT@def@citea{\def\@citea{\NAT@separator}}
\allowdisplaybreaks

\begin{document}
\vspace*{10mm}

\noindent
{\Large \bf The existence of a $\{P_{2},C_{3},P_{5},\mathcal{T}(3)\}$-factor based on the size or the $A_{\alpha}$-spectral radius of graphs}

\vspace*{7mm}

\noindent
{\large \bf Xianglong Zhang, Lihua You$^*$}
\noindent

\vspace{7mm}

\noindent
School of Mathematical Sciences, South China Normal University,  Guangzhou, 510631, P. R. China,
e-mail: {\tt 2023021945@m.scnu.edu.cn}, {\tt ylhua@scnu.edu.cn}.\\[2mm]
$^*$ Corresponding author
\noindent

\vspace{7mm}

\noindent
{\bf Abstract} \
\noindent
Let $G$ be a connected graph of order $n$. A $\{P_{2},C_{3},P_{5},\mathcal{T}(3)\}$-factor of $G$ is a spanning subgraph of $G$ such that each component is isomorphic to a member in $\{P_{2},C_{3},P_{5},\mathcal{T}(3)\}$, where $\mathcal{T}(3)$ is a $\{1,2,3\}$-tree. The $A_{\alpha}$-spectral radius of $G$ is denoted by $\rho_{\alpha}(G)$. In this paper, we obtain a lower bound on the size or the $A_{\alpha}$-spectral radius for $\alpha\in[0,1)$ of $G$ to guarantee that $G$ has a $\{P_{2},C_{3},P_{5},\mathcal{T}(3)\}$-factor, and construct an extremal graph to show that the bound on $A_{\alpha}$-spectral radius is optimal.
 \\[2mm]

\noindent
{\bf Keywords:} Size; $\{P_{2},C_{3},P_{5},\mathcal{T}(3)\}$-factor; $A_{\alpha}$-spectral radius

\noindent
{\bf MSC:} 05C50;05C35

\baselineskip=0.30in

\section{Introduction}

 \hspace{1.5em}Let $G$ be an undirected simple and connected graph with vertex set $V(G)$ and edge set $E(G)$. The \textit{order} of $G$ is the number of its vertices, and the \textit{size} is the number of its edges. 

 Let $G$ be a graph of order $n$ with $V(G)=\{v_{1},v_{2},...,v_{n}\}$. The \textit{adjacency matrix} of $G$ is defined as $A(G)=(a_{ij})$, where $a_{ij}=1$ if $v_{i}v_{j}\in E(G)$, and $a_{ij}=0$ otherwise. The degree diagonal matrix is the diagonal matrix of vertex degrees of $G$, denoted by $D(G)$. The  $signless$ $Laplacian$ $matrix$ $Q(G)$ of $G$ is defined by $Q(G)=D(G)+A(G)$. The largest eigenvalue of $Q(G)$ is denoted by $q(G)$. For any $\alpha \in[0,1)$, Nikiforov\cite{Ba} introduced the $A_{\alpha}$-matrix of $G$ as $A_{\alpha}(G)=\alpha D(G)+(1-\alpha) A(G)$. It is easy to see that  $A_{\alpha}(G)=A(G)$ if $\alpha=0$, and $A_{\alpha}(G)=\frac{1}{2}Q(G)$ if $\alpha=\frac{1}{2}$. The eigenvalues of $A_{\alpha}(G)$ are called the $A_{\alpha}$-eigenvalues of $G$, and the largest of them, denoted by $\rho_{\alpha}(G)$, is called the $A_{\alpha}$-spectral radius of $G$. More interesting spectral properties of $A_{\alpha}(G)$ can be found in \cite{Ba,Bl,Bm}. 

 For a given subset $S\subseteq V(G)$, the subgraph of $G$ induced by $S$ is denoted by $G[S]$, and the subgraph obtained from $G$ by deleting $S$ together with those edges incident to $S$ is denoted by $G-S$. Let $G$ and $H$ be two disjoint graphs. The union $G\cup H$ is the graph with vertex set $V(G)\cup V(H)$ and edge set $E(G)\cup E(H)$. The join $G\vee H$ is derived from $G\cup H$ by joining every vertex of $G$ with every vertex of $H$ by an edge.  
 
 A subgraph of a graph $G$ is spanning if the subgraph covers all vertices of $G$. Let $\mathcal{H}$ be a set of connected graphs. An $\mathcal{H}$-factor of a graph $G$ is a spanning subgraph of $G$, in which each component is isomorphic to an element of $\mathcal{H}$. An $\mathcal{H}$-factor is also referred as a component factor.
 
 For a tree $T$, every vertex of degree 1 is a leaf of $T$. We denote the set of leaves in $T$ by $Leaf(T)$. An edge of $T$ incident with a leaf is called a pendant edge. A $\{1,3\}$-tree is a tree with every vertex having degree 1 or 3. Let $R$ be a $\{1,3\}$-tree, $\mathbb{T}(3)$ be the set of trees $T_{R}$ that can be obtained from $R$ as follows (see \cite{Bkk}): $T_{R}$ is obtained from $R$ by inserting a new vertex of degree 2 into every edge of $R$, and by adding a new pendant edge to every leaf of $R$. Then the tree $T_{R}$ is a $\{1,2,3\}$-tree having $|E(R)|+|Leaf(R)|$ vertices of degree 2 and has the same number of leaves as $R$. The collection of such $\{1,2,3\}$-trees $T_{R}$ generated from all $\{1,3\}$-trees $R$ is denoted by $\mathbb{T}(3)$, and any graph in $\mathbb{T}(3)$ is denoted by $\mathcal{T}(3)$.

  More and more researchers have been studied the existence of different factors in graphs since 2000. Las Vergnas \cite{Be} presented a sufficient and necessary condition for a graph having $\{K_{1,j}:1\leq j\leq k\}$-factor with $k\geq2$. Kano, Lu and Yu \cite{Bg} showed that a graph has a $\{K_{1,2},K_{1,3},K_{5}\}$-factor if it satisfies $i(G-S)\leq\frac{|S|}{2}$ for every $S\subset V(G)$. Kano, Lu and Yu \cite{Bkk} proved that a graph has a $\{P_{2},C_{3},P_{5},\mathcal{T}(3)\}$-factor if and only if it satisfies $i(G-S)\leq\frac{3}{2}|S|$ for all $S\subset V(G)$. Kano and Saito \cite{Bf} proved that a graph $G$ has a $\{K_{1,l}:m\leq l\leq 2m\}$-factor if it satisfies $i(G-S)\leq\frac{1}{m}|S|$ for every $S\subset V(G)$. Zhang, Yan and Kano \cite{Bh} gave a sufficient condition for a graph $G$ containing a $\{K_{1,t}:m\leq t\leq 2m-1\}\cup\{K_{2m+1}\}$-factor. Chen, Lv and Li \cite{Bb} provided a lower  bound on the size (resp. the spectral radius) of $G$ to guarantee that the graph has a $\{P_{2},C_{n}:n\geq3\}$-factor. Lv, Li and Xu \cite{Bbb} derived a tight $A_{\alpha}$-spectral radius and distance signless Laplacian spectral radius for the existence of a $\{K_{2},C_{2i+1}:i\geq1\}$-factor in a graph. Li and Miao \cite{Bo} determined a sufficient condition about the size or the spectral radius of $G$ to contain $\mathcal{P}_{\geq2}$-factor and be $\mathcal{P}_{\geq2}$-factor covered graphs. Miao and Li \cite{Bc} showed a lower bound on the size or the spectral radius, and an upper bound on the distance spectral radius of $G$ to ensure that $G$ has a $\{K_{1,j}:1\leq j\leq k\}$-factor. Zhou, Zhang and Sun \cite{Bd} established a relationship between $P_{\geq2}$-factor and $A_{\alpha}$-spectral radius of a graph.  Zhang and You \cite{Bp} showed sufficient conditions via the size or the $A_{\alpha}$-spectral radius of graphs to ensure that a graph contains a $\{K_{1,2},K_{1,3},K_{5}\}$-factor. Zhou \cite{Bn} got a spectral radius condition on the existence of $\{P_{2},C_{3},P_{5},\mathcal{T}(3)\}$-factor in graphs.
 
  Motived by \cite{Bkk,Bn,Bd} directly, it is natural and interesting to study some sufficient and necessary conditions to ensure that a graph contains a $\{P_{2},C_{3},P_{5},\mathcal{T}(3)\}$-factor. In this paper, we focus on the sufficient conditions via the size or the $A_{\alpha}$-spectral radius of graphs and obtain the following two results.

\begin{theorem}\label{c1}
	Let $G$ be a connected graph of order $n\geq 5$, and
	\begin{equation*}
		F(n)=\begin{cases}
			\binom{n-2}{2}+2,&\text{if $n\geq5$ and $n\notin\{6,8\}$};\\
			9,&\text{if $n=6$};\\
			18,&\text{if $n=8$}.\\
		\end{cases}
	\end{equation*} 
  If $|E(G)|>F(n)$, then $G$ has a $\{P_{2},C_{3},P_{5},\mathcal{T}(3)\}$-factor.\\

\end{theorem}
\vspace{-30pt}

\begin{theorem}\label{c2}
	Let $\alpha \in [0,1)$, $\varphi(x)=x^3-((\alpha+1)n+\alpha-4)x^2+(\alpha n^2+(\alpha^2-2\alpha-1)n-2\alpha+1)x-\alpha^2n^2+(5\alpha^2-3\alpha+2)n-10\alpha^2+15\alpha-8$, $G$ be a connected graph of order $n$ with $n\geq f(\alpha)$, where
	\begin{equation*}
		f(\alpha)=\begin{cases}
			20,&\text{if $\alpha\in [0,\frac{1}{2}]$};\\
			25,&\text{if $\alpha\in (\frac{1}{2},\frac{5}{7}]$};\\
			\frac{7}{1-\alpha}+3,&\text{if $\alpha\in (\frac{5}{7},1)$}.\\
		\end{cases}
	\end{equation*}
	 If $\rho_{\alpha}(G)>\tau(n)$, then $G$ has a
	$\{P_{2},C_{3},P_{5},\mathcal{T}(3)\}$-factor, where $\tau(n)$ is the largest root of $\varphi(x)=0$.
\end{theorem}

	Let $\alpha=0$ in Theorem \ref{c2}, the main result of \cite{Bn} via the spectral radius can be obtained, and let  $\alpha=\frac{1}{2}$, we have the following corollary via the signless Laplacian spectral radius immediately.
	
\begin{corollary}\label{c3}
	Let $G$ be a connected graph of order $n$ with $n\geq 20$. If  $q(G)>2\mu(n)$, then $G$ has a $\{P_{2},C_{3},P_{5},\mathcal{T}(3)\}$-factor, where $\mu(n)$ is the largest root of $4x^3-(6n-14)x^2+(2n^2-7n)x-n^2+7n-6=0$.
\end{corollary}

\section{Preliminaries}\label{sec-pre}

\hspace{1.5em}In this section, we introduce some useful definitions and lemmas.

\begin{definition}{\rm(\!\!\cite{Bk})}\label{d1}
	Let $M$ be a complex matrix of order $n$ described in the following
	block form\begin{equation*}
		M=\begin{pmatrix}
			M_{11}&\cdots&M_{1l}\\
			\vdots&\ddots&\vdots\\
			M_{l1}&\cdots&M_{ll}
		\end{pmatrix}
	\end{equation*}
	where the blocks $M_{ij}$ are $n_i \times n_j$ matrices for any $1\leq i, j \leq l$ and $n= n_1 +\cdots + n_l$.
	For $1 \leq i, j \leq l$, let $q_{ij}$ denote the average row sum of $M_{ij}$, i.e. $q_{ij}$ is the sum of all
	entries in $M_{ij}$ divided by the number of rows. Then $Q(M)=(q_{ij})$ (or simply $Q$) is
	called the quotient matrix of $M$. If, in addition, for each pair $i, j$, $M_{ij}$ has a constant row
	sum, i.e., $M_{ij}\vec{e}_{n_j}= q_{ij}\vec{e}_{n_i}$, then $Q$ is called the equitable quotient matrix of $M$, where
	$\vec{e}_k=(1, 1,\cdots, 1)^T \in \mathcal{C}^k$, and $\mathcal{C}$ denotes the field of complex numbers.
	
\end{definition}

Let $M$ be a real nonnegative matrix. The largest eigenvalue of $M$ is called the spectral radius of $M$, denoted by $\rho(M)$.

\begin{lemma}{\rm(\!\!\cite{Bj})}\label{d2}
	Let $B$ be an equitable quotient matrix of $M$ as defined in Definition \ref{d1}, where $M$ is a nonnegative matrix. Then the eigenvalues of $B$ are also eigenvalues of $M$, and $\rho(B)=\rho(M)$.
\end{lemma}

\begin{lemma}{\rm(\!\!\cite{Ba})}\label{d3}
    Let $K_{n}$ be a complete graph of order $n$. Then $\rho_{\alpha}(K_{n})=n-1$.
\end{lemma}  

\begin{lemma}{\rm(\!\!\cite{Ba})}\label{d4}
	If $G$ is a connected graph, and $H$ is a proper subgraph of $G$, then $\rho_{\alpha}(G)>\rho_{\alpha}(H)$.
\end{lemma}

\begin{lemma}{\rm(The Cauchy's interlace theorem \!\!\cite{Bk})}\label{d5}
	Let two sequences of real number, $\lambda_{1}\geq\lambda_{2}\geq\dots\geq\lambda_{n}$ and $\eta_{1}\geq\eta_{2}\geq\dots\eta_{n-1}$, be the eigenvalues of symmetric matrix $A$ and $B$, respectively. If $B$ is a principal submatrix of $A$, then the eigenvalues of $B$ interlace the eigenvalues of $A$, i.e., $\lambda_{1}\geq\eta_{1}\geq\lambda_{2}\geq\dots\geq\eta_{n-2}\geq\lambda_{n-1}\geq\eta_{n-1}\geq\lambda_{n}$.
	
\end{lemma}

Let $i(G)$ denote the number of isolated vertices of $G$. The following lemma gives a sufficient and necessary condition for a graph containing a $\{P_{2},C_{3},P_{5},\mathcal{T}(3)\}$-factor.

\begin{lemma}{\rm(\!\!\cite{Bkk})}\label{d6}
	A graph $G$ has a $\{P_{2},C_{3},P_{5},\mathcal{T}(3)\}$-factor if and only if
	$i(G-S)\leq\frac{3}{2}|S|$ for all $S\subseteq V(G)$.
\end{lemma}

 \section{\textbf{The proof of Theorem \ref{c1}}}
 
   \hspace{1.5em}In this section, we prove Theorem \ref{c1}, which gives a sufficient condition via the size of a connected graph to ensure that the graph contains a $\{P_{2},C_{3},P_{5},\mathcal{T}(3)\}$-factor.
  
  \begin{proof}
   Suppose to the contrary that $G$ contains no $\{P_{2},C_{3},P_{5},\mathcal{T}(3)\}$-factor. By Lemma \ref{d6}, there exists a nonempty subset $S$ of $V(G)$ satisfying $i(G-S)>\frac{3}{2}|S|$. 
   
   Choose such a connected graph $G$ of order $n$ so that its size is as large as possible. With the choice of $G$, the induced subgraph $G[S]$ and every connected component of $G-S$ are complete graphs, and $G=G[S]\vee(G-S)$.
  
   Note that there is at most one non-trivial connected component in $G-S$.  Otherwise, we can add edges among all non-trivial connected components to get a bigger non-trival connected component, which contradicts to the choice of $G$. For convenient, let $|S|=s$ and $i(G-S)=i$. We now consider the following two possible cases.
  
   \textbf{Case 1}. $G-S$ has exactly one non-trivial connected component, say $G_{1}$.
   
   In this case, let $|V(G_{1})|=n_{1}\geq2$. Obviously, $i\geq\lfloor\frac{3s}{2}\rfloor+1=
   \begin{cases}
   	\frac{3}{2}s+\frac{1}{2},&\text{if $s$ is odd};\\
   	\frac{3}{2}s+1,&\text{if $s$ is even}.\\
   \end{cases}$ Now we show $i=\lfloor\frac{3s}{2}\rfloor+1$.
   
   If $i\geq\lfloor\frac{3s}{2}\rfloor+2$, let $H_{1}$ be a new graph obtained from $G$ by joining each vertex of $G_{1}$ with one vertex in $V(G-S)\setminus V(G_{1})$ by an edge. Then we have $|E(H_{1})|=|E(G)|+n_{1}>|E(G)|$ and $i(H_{1}-S)=i-1\geq\lfloor\frac{3s}{2}\rfloor+1$, a contradiction with the choice of $G$. Hence $i=\lfloor\frac{3s}{2}\rfloor+1$ by $i>\frac{3}{2}s$ and $G=K_{s}\vee(K_{n_{1}}\cup(\lfloor\frac{3s}{2}\rfloor+1)K_{1})$.
   
   Clearly, we have $n=s+\lfloor\frac{3s}{2}\rfloor+1+n_{1}\geq
   \begin{cases}
   	\frac{5}{2}s+\frac{5}{2}\geq5,&\text{if $s$ is odd}\\
   	\frac{5}{2}s+3\geq8,&\text{if $s$ is even}\\
   \end{cases}$
   and $|E(G)|=s(\lfloor\frac{3s}{2}\rfloor+1)+\binom{n-\lfloor\frac{3s}{2}\rfloor-1}{2}$. 
   
   Now we show $|E(G)|\leq F(n)$. By $\binom{n-2}{2}+2=
   \begin{cases}
   	8<9,&\text{if $n=6$,}\\
   	17<18,&\text{if $n=8$,}\\
   \end{cases}$ 
   we only need show $|E(G)|\leq\binom{n-2}{2}+2$.
   
   \textbf{Subcase 1.1}. $s$ is odd.
   \begin{equation*}
   	\begin{aligned}
   		\binom{n-2}{2}+2-|E(G)|&=\frac{1}{8}(s-1)(12n-21s-37)\geq\frac{1}{8}(s-1)(9s-7)\geq0.
   	\end{aligned}
   \end{equation*}
   Therefore, $|E(G)|\leq\binom{n-2}{2}+2$ for odd $s$, which is a contradiction. 
   
   \textbf{Subcase 1.2}. $s$ is even.
   \begin{equation*}
   	\begin{aligned}
   		\binom{n-2}{2}+2-|E(G)|&=\frac{1}{8}(-21s^2-26s+12ns-8n+32)\geq\frac{1}{8}(9(s-\frac{5}{9})^2+\frac{47}{9})>0.
   	\end{aligned}
   \end{equation*}
   Therefore, $|E(G)|<\binom{n-2}{2}+2$ for even $s$, which is a contradiction.
   
   Combining the above two subcases, we have $|E(G)|\leq F(n)$, a contradiction.
     
   \textbf{Case 2}. $G-S$ has no non-trivial connected component.
   
   In this case, we prove $i\leq\lfloor\frac{3s}{2}\rfloor+2$ firstly. 
   
   If $i\geq\lfloor\frac{3s}{2}\rfloor+3$, let $H_{2}$ be a new graph obtained from $G$ by adding an edge between two vertices in $V(G-S)$. Clearly, $i(H_{2}-S)=i-2\geq\lfloor\frac{3s}{2}\rfloor+1$ and $H_{2}-S$ has exactly one non-trivial connected component. Together with $|E(G)|<|E(H_{2})|$, we obtain a contradiction with the choice of $G$, which implies $i=\lfloor\frac{3s}{2}\rfloor+1$ or $i=\lfloor\frac{3s}{2}\rfloor+2$ by $i>\frac{3}{2}s$.
   
   \textbf{Subcase 2.1}. $i=\lfloor\frac{3s}{2}\rfloor+1$.
   
   In this subcase, we have $G=K_{s}\vee((\lfloor\frac{3s}{2}\rfloor+1)K_{1})$. Therefore, $n=s+\lfloor\frac{3s}{2}\rfloor+1$, $|E(G)|=\binom{s}{2}+s(\lfloor\frac{3s}{2}\rfloor+1)$, $s\geq2$ by $n\geq5$, and
   \begin{equation*}
   	\begin{aligned}
   		\binom{n-2}{2}+2-|E(G)|=
   		\begin{cases}
   			\frac{1}{8}(s-1)(9s-31),&\text{if $s$ is odd};\\
   			\frac{1}{8}(9s^2-34s+24),&\text{if $s$ is even}.\\
   		\end{cases}
   	\end{aligned}
   \end{equation*}

   If $s$ is odd, we have $|E(G)|\leq\binom{n-2}{2}+2$ for $s\geq5$ (which implies $n\geq13$), and $|E(G)|=2s^2=18$ for $s=3$ (which implies $n=8$).
   
   If $s$ is even, we have $|E(G)|<\binom{n-2}{2}+2$ for $s\geq4$ (which implies $n\geq11$),  and $|E(G)|=9$ for $s=2$ (which implies $n=6$).
   
   Combining the above arguments, we have $|E(G)|\leq F(n)$ for all $n\geq5$, a contradiction.
   
   \textbf{Subcase 2.2}. $i=\lfloor\frac{3s}{2}\rfloor+2$.
   
   In this subcase, we have $G=K_{s}\vee((\lfloor\frac{3s}{2}\rfloor+2)K_{1})$. Therefore, $n=s+\lfloor\frac{3s}{2}\rfloor+2$, $|E(G)|=\binom{s}{2}+s(\lfloor\frac{3s}{2}\rfloor+2)$, $s\geq2$ by $n\geq5$, and
   \begin{equation*}
   	\begin{aligned}
   		\binom{n-2}{2}+2-|E(G)|=
   		\begin{cases}
   			\frac{1}{8}(s-1)(9s-19),&\text{if $s$ is odd};\\
   			\frac{1}{8}(9s^2-22s+16),&\text{if $s$ is even}.\\
   		\end{cases}
   	\end{aligned}
   \end{equation*}
   
   If $s$ is odd, we have $|E(G)|\leq\binom{n-2}{2}+2$ for $s\geq3$ (which implies $n\geq9$).
   
   If $s$ is even, we have $|E(G)|\leq\binom{n-2}{2}+2$ for $s\geq2$ (which implies $n\geq7$). 
   
   Combining the above arguments, we have $|E(G)|\leq F(n)$ for all $n\geq5$, a contradiction.
   
   By Case 1 and Case 2, we complete the proof.
   \end{proof}

 \section{\textbf{The proof of Theorem \ref{c2}}}
 
   \hspace{1.5em}In this section, we prove Theorem \ref{c2}, which presents a sufficient condition, in terms of the $A_{\alpha}$-spectral radius, to determine whether a graph contains a $\{P_{2},C_{3},P_{5},\mathcal{T}(3)\}$-factor or not.
   
   \begin{proof}
   Suppose to the contrary that $G$ does not contain a $\{P_{2},C_{3},P_{5},\mathcal{T}(3)\}$-factor. By Lemma \ref{d6}, there exists a nonempty subset $S$ of $V(G)$ satisfying $i(G-S)>\frac{3}{2}|S|$. 
   
   Choose such a connected graph $G$ of order $n$ so that its  $A_{\alpha}$-spectral radius is as large as possible. Together with Lemma \ref{d4} and the choice of $G$, the induced subgraph $G[S]$ and every connected component of $G-S$ are complete graphs, and $G=G[S]\vee(G-S)$.
   
   It is easy to see that $G-S$ admits at most one non-trivial connected component. Otherwise, we can construct a new graph $G'$ by adding edges among all non-trivial connected components to obtain a bigger non-trivial connected component. Clearly, $G$ is a proper subgraph of $G'$. According to Lemma \ref{d4}, $\rho_{\alpha}(G')>\rho_{\alpha}(G)$, which contradicts the choice of $G$. For convenient, let $|S|=s$ and $i(G-S)=i$.
  
    Now, we show Theorem \ref{c2} by considering the following two cases.
   
    \textbf{Case 1}. $G-S$ has exactly one non-trivial connected component.
    
     In this case, $G=K_{s}\vee(K_{n_{1}}\cup iK_{1})$, where $n_{1}=n-s-i\geq2$. Now we show $i=\lfloor\frac{3s}{2}\rfloor+1$.
    
    If $i\geq\lfloor\frac{3s}{2}\rfloor+2$, then we construct a new graph $G''$ obtained from $G$ by joining each vertex of $K_{n_{1}}$ with one vertex in $iK_{1}$ by an edge. It is obvious that $i(G''-S)=i-1\geq\lfloor\frac{3s}{2}\rfloor+1$ and $G$ is a proper subgraph of $G''$. According to Lemma \ref{d4}, $\rho_{\alpha}(G'')>\rho_{\alpha}(G)$, which contradicts with the choice of $G$. Therefore, $i=\lfloor\frac{3s}{2}\rfloor+1$, $G=K_{s}\vee(K_{n-s-\lfloor\frac{3s}{2}\rfloor-1} \cup (\lfloor\frac{3s}{2}\rfloor+1)K_{1})$ by $i>\frac{3s}{2}$, and the quotient matrix of $A_{\alpha}(G)$ in terms of the partition $\{V((\lfloor\frac{3s}{2}\rfloor+1)K_{1}),V(K_{n-s-\lfloor\frac{3s}{2}\rfloor-1}),V(K_{s})\}$ can be written as
    \begin{equation*}
    	B_{1}=	\begin{pmatrix}
    		\alpha s & 0 & (1-\alpha)s\\
    		0 & n+(\alpha s-s-\lfloor\frac{3s}{2}\rfloor)-2 & (1-\alpha)s\\
    		(1-\alpha)(\lfloor\frac{3s}{2}\rfloor+1) & (1-\alpha)(n-s-\lfloor\frac{3s}{2}\rfloor-1) & \alpha n-\alpha s+s-1\\
    	\end{pmatrix},
    \end{equation*}
    then the characteristic polynomial of $B_{1}$ is 
    \begin{equation}\label{e1}
    	\begin{aligned}
    		f_{B_{1}}(x)&=x^3-((\alpha+1)n+\alpha s-\lfloor\frac{3s}{2}\rfloor-3)x^2\\
    		&-((\alpha n+s-1)\lfloor\frac{3s}{2}\rfloor-\alpha n^2-(\alpha^2+\alpha)sn+(2\alpha+1)n+(2\alpha+1)s-2)x\\
    		&-((2\alpha^2-3\alpha+1)\lfloor\frac{3s}{2}\rfloor+2\alpha^2-3\alpha+1)s^2\\
    		&-((\alpha^2-2\alpha+1)\lfloor\frac{3s}{2}\rfloor^2-((2\alpha^2-2\alpha+1)n-3\alpha^2+5\alpha-3)\lfloor\frac{3s}{2}\rfloor\\
    		&+\alpha^2 n^2-(3\alpha^2-\alpha+1)n+2\alpha^2-2\alpha+2)s.
    	\end{aligned}
    \end{equation}
    
    By Lemma \ref{d2}, $\rho_{\alpha}(G)$ is the largest root of $f_{B_{1}}(x)=0$, say, $f_{B_{1}}(\rho_{\alpha}(G))=0$. Let $\eta_{1}=\rho_{\alpha}(G)\geq\eta_{2}\geq\eta_{3}$ be the three roots of $f_{B_{1}}(x)=0$ and $Q=diag(\lfloor\frac{3s}{2}\rfloor+1,n-s-\lfloor\frac{3s}{2}\rfloor-1,s)$. It is easy to check that
    \begin{align}
    	&Q^\frac{1}{2}B_{1}Q^{-\frac{1}{2}}\nonumber\\=&   	
    	\begin{pmatrix}
    		\alpha s & 0 & (1-\alpha)s^\frac{1}{2}(\lfloor\frac{3s}{2}\rfloor+1)^\frac{1}{2}\\
    		0 & n+(\alpha s-s-\lfloor\frac{3s}{2}\rfloor)-2 & (1-\alpha)s^\frac{1}{2}(n-s-\lfloor\frac{3s}{2}\rfloor-1)^\frac{1}{2}\\
    		(1-\alpha)s^\frac{1}{2}(\lfloor\frac{3s}{2}\rfloor+1)^\frac{1}{2} & (1-\alpha)s^\frac{1}{2}(n-s-\lfloor\frac{3s}{2}\rfloor-1)^\frac{1}{2} & \alpha n-\alpha s+s-1\nonumber\\
    	\end{pmatrix}
    \end{align}
    is symmetric, and contains
    \begin{equation*}
    	\begin{pmatrix}
    		\alpha s & 0\\
    		0 & n+(\alpha s-s-\lfloor\frac{3s}{2}\rfloor)-2\\
    	\end{pmatrix}
    \end{equation*}
    as a submatrix. Since $Q^\frac{1}{2}B_{1}Q^{-\frac{1}{2}}$ and $B_{1}$ admit the same eigenvalues, according to Lemma \ref{d5}, we get
    \begin{equation}\label{e2}
    	\begin{aligned}
    		\alpha s\leq\eta_{2}\leq n+(\alpha s-s-\lfloor\frac{3s}{2}\rfloor)-2<
    		\begin{cases}
    			n-3,&\text{if $s$ is odd},\\
    			n-5,&\text{if $s$ is even}.\\
    		\end{cases}
    	\end{aligned}
    \end{equation}
    
    \textbf{Subcase 1.1}. $s$ is odd.
    
    Let $\varphi(x)=x^3-((\alpha+1)n+\alpha-4)x^2+(\alpha n^2+(\alpha^2-2\alpha-1)n-2\alpha+1)x-\alpha^2n^2+(5\alpha^2-3\alpha+2)n-10\alpha^2+15\alpha-8$, $\tau(n)$ be the largest root of $\varphi(x)=0$, $G_{2}=K_{1}\vee(K_{n-3}\cup 2K_{1})$. Clearly, $f_{B_{1}}(x)=\varphi(x)$ by plugging $s=1$ into (\ref{e1}), then $G\cong G_{2}$, and $\tau(n)=\rho_{\alpha}(G_{2})$ by Lemma \ref{d2}.
    
    Since $K_{n-2}$ is a proper subgraph of $G_{2}$, by Lemmas \ref{d3}-\ref{d4} and (\ref{e2}), we have 
    \begin{equation}\label{e44}
    	\begin{aligned}
    		\tau(n)=\rho_{\alpha}(G_{2})>\rho_{\alpha}(K_{n-2})=n-3>\eta_{2}.
    	\end{aligned}
    \end{equation}
    
    \textbf{Subcase 1.1.1}. $s=1$.
    
    In this subcase, we have $\rho_{\alpha}(G)=\rho_{\alpha}(G_{2})=\tau(n)$, which contradicts with $\rho_{\alpha}(G)>\tau(n)$. 
    
    \textbf{Subcase 1.1.2}. $s\geq3$.
    
    In this subcase, $G\ncong G_{2}$ and $f_{B_{1}}(x)\neq\varphi(x)$. By (\ref{e1}) and $\varphi(\tau(n))=0$, we have
    \begin{equation}\label{e55}
    	\begin{aligned}
    		f_{B_{1}}(\tau(n))&=f_{B_{1}}(\tau(n))-\varphi(\tau(n))=-\frac{1}{4}(s-1)H(\tau(n)),
    	\end{aligned}
    \end{equation}
    where $H(x)=(4\alpha-6)x^2+(-4\alpha^2n+2\alpha n+6s+8\alpha+2)x+4\alpha^2n^2-(12\alpha^2-12\alpha+6)sn-(20\alpha^2-12\alpha+8)n+(21\alpha^2-36\alpha+15)s^2+(37\alpha^2-60\alpha+29)s+40\alpha^2-60\alpha+32$, its axis symmetry is $x=-\frac{-4\alpha^2n+2\alpha n+6s+8\alpha+2}{2(4\alpha-6)}$. Now we show $\rho_{\alpha}(G)<\tau(n)$, and we obtain a contradiction with $\rho_{\alpha}(G)>\tau(n)$.
    
    \textbf{Subcase 1.1.2.1}. $0\leq\alpha\leq\frac{5}{7}$.
    
    Firstly, we show $H(\tau(n))<H(n-3)$.
    
    By (\ref{e44}), we only need show $-\frac{-4\alpha^2n+2\alpha n+6s+8\alpha+2}{2(4\alpha-6)}<n-3$, say, show $g_{1}(s)=2(n-3)(4\alpha-6)+(-4\alpha^2n+2\alpha n+6s+8\alpha+2)<0$. 
    
    In fact, by $n=n_{1}+\frac{5}{2}s+\frac{1}{2}\geq\frac{5}{2}s+\frac{5}{2}$ and $s\geq3$, we have 
    \begin{equation*}
    	\begin{aligned}
    		g_{1}(s)&=(-4\alpha^2+10\alpha-12)n-16\alpha+6s+38\\
    		&\leq(-4(\alpha-\frac{5}{4})^2-\frac{23}{4})(\frac{5}{2}s+\frac{5}{2})-16\alpha+6s+38\\
    		&=s(-10\alpha^2+25\alpha-24)-10\alpha^2+9\alpha+8\\
    		&\leq3(-10\alpha^2+25\alpha-24)-10\alpha^2+9\alpha+8\\
    		&<0.\\
    	\end{aligned}
    \end{equation*}
    
    By the above arguments and (\ref{e44}), we have $H(\tau(n))<H(n-3)$.
    
    On the other hand, by a directed calculation, we have $H(n-3)=(6\alpha-6)n^2+(-12s\alpha^2+12s\alpha-8\alpha^2-10\alpha+30)n+(21\alpha^2-36\alpha+15)s^2+(37\alpha^2-60\alpha+11)s+40\alpha^2-48\alpha-28$. Let $P(n)=H(n-3)$. Then the axis symmetry of $P(x)$ is $x=-\frac{-12s\alpha^2+12s\alpha-8\alpha^2-10\alpha+30}{2(6\alpha-6)}$.
    
    Now we show  $P(n)<0$. Let $g_{2}(s)=2(\frac{5}{2}s+\frac{5}{2})(6\alpha-6)+(-12s\alpha^2+12s\alpha-8\alpha^2-10\alpha+30)$. Then by $\alpha\leq\frac{5}{7}$, we have
    \begin{equation*}
    	\begin{aligned}
    		g_{2}(s)&=(-12\alpha^2+42\alpha-30)s-8\alpha^2+20\alpha\\
    		&=(-12(\alpha-\frac{7}{4})^2+\frac{27}{4})s-8(\alpha-\frac{5}{4})^2+\frac{25}{2}\\
    		&<-\frac{300}{49}s+\frac{25}{2}\\
    		&<0,\\
    	\end{aligned}
    \end{equation*}
    which implies $-\frac{-12s\alpha^2+12s\alpha-8\alpha^2-10\alpha+30}{2(6\alpha-6)}<\frac{5}{2}s+\frac{5}{2}$, and $P(n)$ is monotonically decreasing when $n\geq\frac{5}{2}s+\frac{5}{2}$. Then by $-\frac{\frac{63}{2}s^2+20s-\frac{71}{2}}{2(-9s^2-13s+20)}>\frac{5}{7}\geq\alpha$, we have
    \begin{equation}\label{e55}
    	\begin{aligned}
    		P(n)&\leq P(\frac{5}{2}s+\frac{5}{2})\\
    		&=(6\alpha-6)(\frac{5}{2}s+\frac{5}{2})^2+(-12s\alpha^2+12s\alpha-8\alpha^2-10\alpha+30)(\frac{5}{2}s+\frac{5}{2})\\
    		&+(21\alpha^2-36\alpha+15)s^2+(37\alpha^2-60\alpha+11)s+40\alpha^2-48\alpha-28\\
    		&=(-9s^2-13s+20)\alpha^2+(\frac{63}{2}s^2+20s-\frac{71}{2})\alpha-\frac{45}{2}s^2+11s+\frac{19}{2}\\
    		&\leq\frac{25}{49}(-9s^2-13s+20)+\frac{5}{7}(\frac{63}{2}s^2+20s-\frac{71}{2})-\frac{45}{2}s^2+11s+\frac{19}{2}\\
    		&=-\frac{1}{49}(225s^2-914s+277).
    	\end{aligned}
    \end{equation}
    
    If $s\geq5$, then $P(n)<0$ by (\ref{e55}).
    
    If $s=3$, then $-\frac{-12s\alpha^2+12s\alpha-8\alpha^2-10\alpha+30}{2(6\alpha-6)}=-\frac{-44\alpha^2+26\alpha+30}{2(6\alpha-6)}<8<20\leq f(\alpha)\leq n$ due to $0\leq\alpha\leq\frac{5}{7}$. Hence, 
    \begin{equation*}
    	\begin{aligned}
    		p(n)\leq
    		\begin{cases}
    			p(20)=-540\alpha^2+2368\alpha-1660<0,&\text{if $0\leq\alpha\leq\frac{1}{2}$ and $n\geq f(\alpha)=20$},\\
    			p(25)=-760\alpha^2+3848\alpha-2860<0,&\text{if $\frac{1}{2}<\alpha\leq\frac{5}{7}$ and $n\geq f(\alpha)=25$}.\\
    		\end{cases}
    	\end{aligned}
    \end{equation*}

    Thus, we conclude that $P(n)<0$ for $s\geq3$. Combining the above arguments, by (\ref{e55}), we have $f_{B_{1}}(\tau(n))=-\frac{1}{4}(s-1)H(\tau(n))>-\frac{1}{4}(s-1)H(n-3)=-\frac{1}{4}(s-1)P(n)>0$, which implies $\rho_{\alpha}(G)<\tau(n)$ for $s\geq3$ by (\ref{e44}), a contradiction with $\rho_{\alpha}(G)>\tau(n)$.

    \textbf{Subcase 1.1.2.2}. $\frac{5}{7}<\alpha<1$.
    
    By (\ref{e1}), we have
    \begin{equation*}
    	\begin{aligned}
    		f_{B_{1}}(n-3)&=\frac{3}{4}(7\alpha-5)(1-\alpha)s^3+((3\alpha^2-3\alpha)n-4\alpha^2+6\alpha+1)s^2\\
    		&+((\frac{3}{2}-\frac{3}{2}\alpha)n^2-(\alpha^2-\frac{11}{2}\alpha+\frac{15}{2})n-\frac{3}{4}\alpha^2-3\alpha+\frac{39}{4})s\\
    		&+(\frac{3}{2}\alpha-\frac{3}{2})n^2-(\frac{9}{2}\alpha-\frac{15}{2})n-9.\\
    	\end{aligned}
    \end{equation*}
    
    Let $\Psi(s,n)=f_{B_{1}}(n-3)$.
    Thus, we get
    \begin{equation*}
    	\begin{aligned}
    		\frac{\partial \Psi(s,n)}{\partial s}
    		&=\frac{9}{4}(7\alpha-5)(1-\alpha)s^2+2((3\alpha^2-3\alpha)n-4\alpha^2+6\alpha+1)s\\
    		&+((\frac{3}{2}-\frac{3}{2}\alpha)n^2-(\alpha^2-\frac{11}{2}\alpha+\frac{15}{2})n-\frac{3}{4}\alpha^2-3\alpha+\frac{39}{4}).
    	\end{aligned}
    \end{equation*}
    
    Since $\frac{5}{7}<\alpha<1$, $n\geq f(\alpha)>\frac{7}{1-\alpha}+3>\frac{7}{1-\alpha}$. By a simple computation, we have 
    \begin{equation*}
    	\begin{aligned}
    		\left.\frac{\partial \Psi(s,n)}{\partial s}\right|_{s=3}
    		&=(\frac{3}{2}-\frac{3}{2}\alpha)n^2+(17\alpha^2-\frac{25}{2}\alpha-\frac{15}{2})n-\frac{333}{2}\alpha^2+276\alpha-\frac{171}{2}\\
    		&>(\frac{3}{2}-\frac{3}{2}\alpha)(\frac{7}{1-\alpha})^2+(17\alpha^2-\frac{25}{2}\alpha-\frac{15}{2})(\frac{7}{1-\alpha})-\frac{333}{2}\alpha^2+276\alpha-\frac{171}{2}\\
    		&=\frac{1}{2(1-\alpha)}(333\alpha^3-647\alpha^2+548\alpha-129)\\
    		&>0,
    	\end{aligned}
    \end{equation*}
    and
    \begin{equation*}
    	\begin{aligned}
    		\left.\frac{\partial \Psi(s,n)}{\partial s}\right|_{s=\frac{2}{5}n-1}
    		&=\frac{1}{50}((-6\alpha^2+21\alpha-15)n^2+(120\alpha^2-265\alpha+115)n-425\alpha^2+600\alpha-175)\\
    		&<\frac{1}{50}((-6\alpha^2+21\alpha-15)(\frac{7}{1-\alpha})^2+(120\alpha^2-265\alpha+115)(\frac{7}{1-\alpha})\\
    		&-425\alpha^2+600\alpha-175)\\
    		&=\frac{1}{50(1-\alpha)}(425\alpha^3-185\alpha^2-786\alpha-105)\\
    		&<0.
    	\end{aligned}
    \end{equation*}
    
    Then $f_{B_{1}}(n-3)=\Psi(s,n)\geq$ min$\{\Psi(3,n),\Psi(\frac{2}{5}n-1,n)\}$ since the leading coefficient of $\Psi(s,n)$ (when viewed as a cubic polynomial of $s$) is positive, and $3\leq s\leq\frac{2}{5}n-1$. 
    
    By $\frac{5}{7}<\alpha<1$ and $n\geq f(\alpha)>\frac{7}{1-\alpha}+3>\frac{7}{1-\alpha}$, we have
    \begin{equation*}
    	\begin{aligned}
    		\Psi(3,n)&=(3-3\alpha)n^2+(24\alpha^2-15\alpha-15)n-180\alpha^2+288\alpha-72\\
    		&>(3-3\alpha)(\frac{7}{1-\alpha})^2+(24\alpha^2-15\alpha-15)(\frac{7}{1-\alpha})-180\alpha^2+288\alpha-72\\
    		&=\frac{1}{1-\alpha}(180\alpha^3-300\alpha^2+255\alpha-30)\\
    		&>0,
    	\end{aligned}
    \end{equation*}
    and
    \begin{equation*}
    	\begin{aligned}
    		\Psi(\frac{2}{5}n-1,n)&=\frac{1}{125}((18\alpha^2-63\alpha+45)n^3-(115\alpha^2-530\alpha+505)n^2\\
    		&+(75\alpha^2-1025\alpha+1700)n+250\alpha^2-1750)\\
    		&>\frac{1}{125}((18\alpha^2-63\alpha+45)(\frac{7}{1-\alpha})^3-(115\alpha^2-530\alpha+505)(\frac{7}{1-\alpha})^2\\
    		&+(75\alpha^2-1025\alpha+1700)(\frac{7}{1-\alpha})+250\alpha^2-1750)\\
    		&=\frac{1}{125(1-\alpha)^2}(250\alpha^4-1025\alpha^3+565\alpha^2+4221\alpha+840)\\
    		&>0.
    	\end{aligned}
    \end{equation*}
    Therefore, we conclude that $f_{B_{1}}(n-3)\geq$ min$\{\Psi(3,n),\Psi(\frac{2}{5}n-1,n)\}>0$ for $s\geq3$. By (\ref{e44}), we have $\rho_{\alpha}(G)<\tau(n)$ for $s\geq3$, which contradicts with $\rho_{\alpha}(G)>\tau(n)$.

    \textbf{Subcase 1.2}. $s$ is even.
    
    Let $\psi(x)=x^3-((\alpha+1)n+2\alpha-6)x^2+(\alpha n^2+(2\alpha^2-3\alpha-1)n-4\alpha-3)x-2\alpha^2n^2+(18\alpha^2-14\alpha+8)n-72\alpha^2+118\alpha-56$, $\theta(n)$ be the largest root of $\psi(x)=0$, $G_{3}=K_{2}\vee(K_{n-6}\cup 4K_{1})$. Clearly, $f_{B_{1}}(x)=\psi(x)$ by plugging $s=2$ into (\ref{e1}), then $G\cong G_{3}$, and $\theta(n)=\rho_{\alpha}(G_{3})$ by Lemma \ref{d2}.
    
    Since $K_{n-4}$ is a proper subgraph of $G_{3}$, by Lemmas \ref{d3}-\ref{d4} and (\ref{e2}),  we have 
    \begin{equation}\label{e4}
    	\begin{aligned}
    		\theta(n)=\rho_{\alpha}(G_{3})>\rho_{\alpha}(K_{n-4})=n-5>\eta_{2}.
    	\end{aligned}
    \end{equation}
    
    \textbf{Subcase 1.2.1}. $s=2$.
    
    In this subcase, we have $\rho_{\alpha}(G)=\rho_{\alpha}(G_{3})=\theta(n)$. Now we prove $\theta(n)<n-3<\tau(n)$ to get the contradiction. 
    
    By a directed computation, we have
    \begin{equation*}
    	\begin{aligned}
    		\psi(n-3)&=(2-2\alpha)n^2+(12\alpha^2-6\alpha-10)n-72\alpha^2+112\alpha-20.
    	\end{aligned}
    \end{equation*}
    
    Let $\Omega(x)=(2-2\alpha)x^2+(12\alpha^2-6\alpha-10)x-72\alpha^2+112\alpha-20$. If $\alpha\in [0,\frac{5}{7}]$, then $x=-\frac{12\alpha^2-6\alpha-10}{2(2-2\alpha)}<8$ for  and $\Omega(x)$ is increasing when $x\geq20$, so $\psi(n-3)=\Omega(n)\geq\Omega(20)=168\alpha^2-808\alpha+580>0$ when $n\geq20$. If $\alpha\in (\frac{5}{7},1)$,  then $-\frac{12\alpha^2-6\alpha-10}{2(2-2\alpha)}<\frac{7}{1-\alpha}$, and $\Omega(n)$ is increasing when $x\geq\frac{7}{1-\alpha}+3>\frac{7}{1-\alpha}$, so $\psi(n-3)=\Omega(n)>\Omega(\frac{7}{1-\alpha})=\frac{72\alpha^3-100\alpha^2+90\alpha+8}{1-\alpha}>0$ when $n\geq\frac{7}{1-\alpha}+3$.
    
    Therefore, we get $\theta(n)<n-3$ by $\psi(n-3)>0$ and (\ref{e4}), then  $\rho_{\alpha}(G)=\rho_{\alpha}(G_{3})=\theta(n)<\tau(n)$ by (\ref{e44}), which contradicts with $\rho_{\alpha}(G)>\tau(n)$.
   
   \textbf{Subcase 1.2.2}. $s\geq4$.
   
   In this subcase, $G\ncong G_{3}$ and $f_{B_{1}}(x)\neq\psi(x)$. By (\ref{e1}) and $\psi(\theta(n))=0$, we have
   \begin{equation}\label{e5}
   	\begin{aligned}
   		f_{B_{1}}(\theta(n))&=f_{B_{1}}(\theta(n))-\psi(\theta(n))=-\frac{1}{4}(s-2)h(\theta(n)),
   	\end{aligned}
   \end{equation}
   where $h(x)=(4\alpha-6)x^2+(-4\alpha^2n+2\alpha n+6s+8\alpha+10)x+4\alpha^2n^2-(12\alpha^2-12\alpha+6)sn-(36\alpha^2-28\alpha+16)n+(21\alpha^2-36\alpha+15)s^2+(68\alpha^2-114\alpha+52)s+144\alpha^2-236\alpha+112$, its axis symmetry is $x=-\frac{-4\alpha^2n+2\alpha n+6s+8\alpha+10}{2(4\alpha-6)}$.
   
   \textbf{Subcase 1.2.2.1}. $0\leq\alpha\leq\frac{5}{7}$.
   
   Firstly, we show $h(\theta(n))<h(n-5)$.
   
   By (\ref{e4}), we only need show $-\frac{-4\alpha^2n+2\alpha n+6s+8\alpha+10}{2(4\alpha-6)}<n-5$, say, show $g_{3}(s)=2(n-5)(4\alpha-6)+(-4\alpha^2n+2\alpha n+6s+8\alpha+10)<0$.
   
   In fact, by $n=n_{1}+\frac{5}{2}s+1\geq\frac{5}{2}s+3$ and $s\geq4$, we have 
   \begin{equation*}
   	\begin{aligned}
   		g_{3}(s)&=(-4\alpha^2+10\alpha-12)n-32\alpha+6s+70\\
   		&\leq(-4(\alpha-\frac{5}{4})^2-\frac{23}{4})(\frac{5}{2}+3)-32\alpha+6s+70\\
   		&=s(-10\alpha^2+25\alpha-30)-12\alpha^2-2\alpha+34\\
   		&\leq4(-10\alpha^2+25\alpha-30)-12\alpha^2-2\alpha+34\\
   		&<0.\\
   	\end{aligned}
   \end{equation*}
   
   By the above arguments and (\ref{e4}), we have $h(\theta(n))<h(n-5)$.
   
   On the other hand, by a directed calculation, we have $h(n-5)=(6\alpha-6)n^2+(-12s\alpha^2+12s\alpha-16\alpha^2-14\alpha+54)n+(21\alpha^2-36\alpha+15)s^2+(68\alpha^2-114\alpha+22)s+144\alpha^2-176\alpha-88$. Let $p(n)=h(n-5)$. Then the axis symmetry of $p(x)$ is $x=-\frac{-12s\alpha^2+12s\alpha-16\alpha^2-14\alpha+54}{2(6\alpha-6)}$.
   
   Now we show  $p(n)<0$.
   
   If $s\geq10$, we take $g_{4}(s)=2(\frac{5}{2}s+3)(6\alpha-6)+(-12s\alpha^2+12s\alpha-16\alpha^2-14\alpha+54)$. By $\alpha\leq\frac{5}{7}$, we have
   \begin{equation*}
   	\begin{aligned}
   		g_{4}(s)&=(-12\alpha^2+42\alpha-30)s-16\alpha^2+22\alpha+18\\
   		&=(-12(\alpha-\frac{7}{4})^2+\frac{27}{4})s-16(\alpha-\frac{11}{16})^2+\frac{409}{16}\\
   		&<-\frac{300}{49}s+\frac{409}{16}\\
   		&<0,\\
   	\end{aligned}
   \end{equation*}
   which implies $-\frac{-12s\alpha^2+12s\alpha-16\alpha^2-14\alpha+54}{2(6\alpha-6)}<\frac{5}{2}s+3$, and $p(n)$ is monotonically decreasing when $n\geq\frac{5}{2}s+3$. Then by $-\frac{\frac{63}{2}s^2-23s-164}{2(-9s^2-8s+96)}>\frac{5}{7}\geq\alpha$, we have
   \begin{equation*}{\label{e7}}
   	\begin{aligned}
   		p(n)&\leq p(\frac{5}{2}s+3)\\
   		&=(6\alpha-6)(\frac{5}{2}s+3)^2+(-12s\alpha^2+12s\alpha-16\alpha^2-14\alpha+54)(\frac{5}{2}s+3)\\
   		&+(21\alpha^2-36\alpha+15)s^2+(68\alpha^2-114\alpha+22)s+144\alpha^2-176\alpha-88\\
   		&=(-9s^2-8s+96)\alpha^2+(\frac{63}{2}s^2-23s-164)\alpha-\frac{45}{2}s^2+67s+20\\
   		&\leq\frac{25}{49}(-9s^2-8s+96)+\frac{5}{7}(\frac{63}{2}s^2-23s-164)-\frac{45}{2}s^2+67s+20\\
   		&=-\frac{1}{49}(225s^2-2278s+2360)\\
   		&<0.
   	\end{aligned}
   \end{equation*}

   If $s\in\{4,6,8\}$, then
   \begin{equation*}
   	\begin{aligned}
   		-\frac{-12s\alpha^2+12s\alpha-16\alpha^2-14\alpha+54}{2(6\alpha-6)}=
   		\begin{cases}
   			-\frac{-112\alpha^2+82\alpha+54}{2(6\alpha-6)}<17,&\text{if $s=8$},\\
   			-\frac{-88\alpha^2+58\alpha+54}{2(6\alpha-6)}<15,&\text{if $s=6$},\\
   			-\frac{-64\alpha^2+34\alpha+54}{2(6\alpha-6)}<14,&\text{if $s=4$}.\\
   		\end{cases}
   	\end{aligned}
   \end{equation*}
   
   For $0\leq\alpha\leq\frac{1}{2}$ and $n\geq f(\alpha)=20$, we have
   \begin{equation*}
   	\begin{aligned}
   		p(n)\leq p(20)=
   		\begin{cases}
   			-208\alpha^2+648\alpha-272\leq0,&\text{if $s=8$},\\
   			-452\alpha^2+1404\alpha-736<0,&\text{if $s=6$},\\
   			-528\alpha^2+1872\alpha-1080<0,&\text{if $s=4$}.\\
   		\end{cases}
   	\end{aligned}
   \end{equation*}
   
   For $\frac{1}{2}<\alpha\leq\frac{5}{7}$ and $n\geq f(\alpha)=25$, we have
   \begin{equation*}
   	\begin{aligned}
   		p(n)\leq p(25)=
   		\begin{cases}
   			-768\alpha^2+2408\alpha-1352<0,&\text{if $s=8$},\\
   			-892\alpha^2+3044\alpha-1816<0,&\text{if $s=6$},\\
   			-848\alpha^2+3392\alpha-2160<0,&\text{if $s=4$}.\\
   		\end{cases}
   	\end{aligned}
   \end{equation*}
   
   Thus, we conclude that $p(n)<0$ for $s\geq4$. Combining the above arguments, by (\ref{e5}), we have $f_{B_{1}}(\theta(n))=-\frac{1}{4}(s-2)h(\theta(n))>-\frac{1}{4}(s-2)h(n-3)=-\frac{1}{4}(s-2)p(n)\geq0$, which implies $\rho_{\alpha}(G)<\theta(n)<\tau(n)$ for $s\geq4$ by (\ref{e4}), a contradiction with $\rho_{\alpha}(G)>\tau(n)$.

   \textbf{Subcase 1.2.2.2}. $\frac{5}{7}<\alpha<1$.
   
  By (\ref{e1}), we have
  \begin{equation*}
  	\begin{aligned}
  		f_{B_{1}}(n-5)&=\frac{3}{4}(7\alpha-5)(1-\alpha)s^3+((3\alpha^2-3\alpha)n-\frac{13}{2}\alpha^2+\frac{21}{2}\alpha+2)s^2\\
  		&+((\frac{3}{2}-\frac{3}{2}\alpha)n^2-(2\alpha^2-\frac{19}{2}\alpha+\frac{27}{2})n-2\alpha^2-13\alpha+33)s\\
  		&+(3\alpha-3)n^2-(15\alpha-27)n-60.\\
  	\end{aligned}
  \end{equation*}
  
  Let $\Phi(s,n)=f_{B_{1}}(n-5)$. Thus, we get
  \begin{equation*}
  	\begin{aligned}
  		\frac{\partial \Phi(s,n)}{\partial s}
  		&=\frac{9}{4}(7\alpha-5)(1-\alpha)s^2+2((3\alpha^2-3\alpha)n-\frac{13}{2}\alpha^2+\frac{21}{2}\alpha+2)s\\
  		&+((\frac{3}{2}-\frac{3}{2}\alpha)n^2-(2\alpha^2-\frac{19}{2}\alpha+\frac{27}{2})n-2\alpha^2-13\alpha+33).
  	\end{aligned}
  \end{equation*}
  
  Since $\frac{5}{7}<\alpha<1$, $n\geq f(\alpha)>\frac{7}{1-\alpha}+3>\frac{7}{1-\alpha}$. By a simple computation, we have 
  \begin{equation*}
  	\begin{aligned}
  		\left.\frac{\partial \Phi(s,n)}{\partial s}\right|_{s=4}
  		&=(\frac{3}{2}-\frac{3}{2}\alpha)n^2+(22\alpha^2-\frac{29}{2}\alpha-\frac{27}{2})n-306\alpha^2+503\alpha-131\\
  		&>(\frac{3}{2}-\frac{3}{2}\alpha)(\frac{7}{1-\alpha})^2+(22\alpha^2-\frac{29}{2}\alpha-\frac{27}{2})(\frac{7}{1-\alpha})-306\alpha^2+503\alpha-131\\
  		&=\frac{1}{2(1-\alpha)}(612\alpha^3-1310\alpha^2+1065\alpha-304)\\
  		&>0,
  	\end{aligned}
  \end{equation*}
  and
  \begin{equation*}
  	\begin{aligned}
  		\left.\frac{\partial \Phi(s,n)}{\partial s}\right|_{s=\frac{2}{5}n-\frac{6}{5}}
  		&=\frac{1}{50}((-6\alpha^2+21\alpha-15)n^2+(36\alpha^2-41\alpha-55)n-454\alpha^2+34\alpha+600)\\
  		&<\frac{1}{50}((-6\alpha^2+21\alpha-15)(\frac{7}{1-\alpha})^2+(36\alpha^2-41\alpha-55)(\frac{7}{1-\alpha})\\
  		&-454\alpha^2+34\alpha+600)\\
  		&=\frac{1}{50(1-\alpha)}(454\alpha^3-236\alpha^2-559\alpha-520)\\
  		&<0.
  	\end{aligned}
  \end{equation*}
  
  Then $f_{B_{1}}(n-5)=\Phi(s,n)\geq$ min$\{\Phi(4,n),\Phi(\frac{2}{5}n-\frac{6}{5},n)\}$ since the leading coefficient of $\Phi(s,n)$ (when viewed as a cubic polynomial of $s$) is positive, and $4\leq s\leq\frac{2}{5}n-\frac{6}{5}$. 
  
  By $\frac{5}{7}<\alpha<1$ and $n\geq f(\alpha)>\frac{7}{1-\alpha}+3>\frac{7}{1-\alpha}$, we have
  \begin{equation*}
  	\begin{aligned}
  		\Phi(4,n)&=(3-3\alpha)n^2+(40\alpha^2-25\alpha-27)n-448\alpha^2+692\alpha-136\\
  		&>(3-3\alpha)(\frac{7}{1-\alpha})^2+(40\alpha^2-25\alpha-27)(\frac{7}{1-\alpha})-448\alpha^2+692\alpha-136\\
  		&=\frac{1}{1-\alpha}(448\alpha^3-860\alpha^2+653\alpha-178)\\
  		&>0,
  	\end{aligned}
  \end{equation*}
  and
  \begin{equation*}
  	\begin{aligned}
  		\Phi(\frac{2}{5}n-\frac{6}{5},n)&=\frac{1}{125}((18\alpha^2-63\alpha+45)n^3-(212\alpha^2-997\alpha+965)n^2\\
  		&+(386\alpha^2-3806\alpha+6000)n+264\alpha^2+1896\alpha-11280)\\
  		&>\frac{1}{125}((18\alpha^2-63\alpha+45)(\frac{7}{1-\alpha}+3)^3-(212\alpha^2-997\alpha+965)(\frac{7}{1-\alpha}+3)^2\\
  		&+(386\alpha^2-3806\alpha+6000)(\frac{7}{1-\alpha}+3)+264\alpha^2+1896\alpha-11280)\\
  		&=\frac{1}{125(1-\alpha)^2}(550\alpha^3-4825\alpha^2+7496\alpha-2780)\\
  		&>0.
  	\end{aligned}
  \end{equation*}
  Therefore, we conclude that $f_{B_{1}}(n-5)\geq$ min$\{\Phi(4,n),\Phi(\frac{2}{5}n-\frac{6}{5},n)\}>0$ for $s\geq4$. By (\ref{e4}), we have $\rho_{\alpha}(G)<\theta(n)<\tau(n)$ for $s\geq4$, which contradicts $\rho_{\alpha}(G)>\tau(n)$.

  	\textbf{Case 2}. $G-S$ has no non-trivial connected component.
  	
  	In this case, $G=K_{s}\vee iK_{1}$. Now we show $i=\lfloor\frac{3s}{2}\rfloor+1$ or $i=\lfloor\frac{3s}{2}\rfloor+2$.

  	If $i\geq\lfloor\frac{3s}{2}\rfloor+3$, then we create a new graph $G'''$ by adding an edge between two vertices in $iK_{1}$. Thus, $i(G'''-S)=i-2\geq\lfloor\frac{3s}{2}\rfloor+1$ and $G'''-S$ admits exactly one non-trivial connected component. Obviously, $G$ is a proper subgraph of $G'''$, and then $\rho_{\alpha}(G)\leq\rho_{\alpha}(G''')\leq\tau(n)$ by applying Case 1, a contradiction. Therefore, $i=\lfloor\frac{3s}{2}\rfloor+1$ or $i=\lfloor\frac{3s}{2}\rfloor+2$ by $i>\frac{3s}{2}$.
  	
  	\textbf{Subcase 2.1}. $i=\lfloor\frac{3s}{2}\rfloor+1$.
  	
  	Obviously, $n=s+\lfloor\frac{3s}{2}\rfloor+1$, and the quotient matrix of $A_{\alpha}(G)$ in view of the partition $\{V((\lfloor\frac{3s}{2}\rfloor+1)K_{1}),V(K_{s})\}$ equals
  	\begin{equation*}
  		B_{2}=	\begin{pmatrix}
  			\alpha s & (1-\alpha)s\\
  			(1-\alpha)(\lfloor\frac{3s}{2}\rfloor+1) & \alpha n-\alpha s+s-1\\
  		\end{pmatrix}.
  	\end{equation*}
  	
  	Then the characteristic polynomial of $B_{2}$ is
  	\begin{equation*}
  		\begin{aligned}
  			f_{B_{2}}(x)&=x^2-(\alpha n+s-1)x+\alpha^2sn-(\alpha^2-2\alpha+1)s\lfloor\frac{3s}{2}\rfloor-(\alpha^2-\alpha)s^2-(\alpha^2-\alpha)s-s\\
  			&=x^2-(\alpha n+s-1)x+(2\alpha-1)s\lfloor\frac{3s}{2}\rfloor+\alpha s^2+\alpha s-s,
  		\end{aligned}
  	\end{equation*}
  	and $\rho_{\alpha}(G)$ is the largest root of  $f_{B_{2}}(x)=0$ by Lemma \ref{d2}. By a simple computation, we have
  	\begin{equation}\label{e999}
  		\begin{aligned}
  			\rho_{\alpha}(G)=\frac{\alpha n+s-1+\sqrt{(\alpha n+s-1)^2-4((2\alpha-1)s\lfloor\frac{3s}{2}\rfloor+\alpha s^2+\alpha s-s)}}{2}.
  		\end{aligned}
  	\end{equation}
  	
  	Now we show $\rho_{\alpha}(G)<n-3$. It follows from $n=s+\lfloor\frac{3s}{2}\rfloor+1$ that
  	\begin{equation}\label{e18}
  		\begin{aligned}
  			&\hspace{1.3em}(2(n-3)-(\alpha n+s-1))^2-(\alpha n+s-1)^2+4((2\alpha-1)s\lfloor\frac{3s}{2}\rfloor+\alpha s^2+\alpha s-s)\\
  			&=(4-4\alpha)n^2-(4s+20-12\alpha)n+(8\alpha-4)s\lfloor\frac{3s}{2}\rfloor+4\alpha s^2+4\alpha s+8s+24\\&=
  			\begin{cases}
  				9(1-\alpha)s^2-4(-5\alpha+8)s+5\alpha+15,&\text{if $s$ is odd},\\
  				9(1-\alpha)s^2-2(-7\alpha+13)s+8\alpha+8,&\text{if $s$ is even}.\\
  			\end{cases}
  		\end{aligned}
  	\end{equation}
  		
  	\textbf{Subcase 2.1.1}. $s$ is odd.
  		
  	Let $t_{1}(s)=9(1-\alpha)s^2-4(-5\alpha+8)s+5\alpha+15$. Then we obtain $t_{1}(9)=456-544\alpha$, $t_{1}(11)=752-864\alpha$ and $t_{1}(\frac{19-5\alpha}{5(1-\alpha)})=\frac{1}{25(1-\alpha)}(-400\alpha^2+740\alpha+584)$.
  	
  	Since 
  	\begin{equation*}
  		\begin{aligned}
  			n=\frac{5}{2}s+\frac{1}{2}\geq
  			\begin{cases}
  				20,&\text{if $\alpha\in[0,\frac{1}{2}]$},\\
  				25,&\text{if $\alpha\in(\frac{1}{2},\frac{5}{7}]$},\\
  				\frac{7}{1-\alpha}+3,&\text{if $\alpha\in(\frac{5}{7},1)$},\\
  			\end{cases}
  		\end{aligned}
  	\end{equation*}
  	we have
  	\begin{equation*}
  		\begin{aligned}
  			\frac{4(-5\alpha+8)}{2(9-9\alpha)}<
  			\begin{cases}
  				3<9\leq s,&\text{if $\alpha\in[0,\frac{1}{2}]$},\\
  				4<11\leq s,&\text{if $\alpha\in(\frac{1}{2},\frac{5}{7}]$},\\
  				\frac{19-5\alpha}{5(1-\alpha)}\leq s,&\text{if $\alpha\in(\frac{5}{7},1)$},\\
  			\end{cases}
  		\end{aligned}
  	\end{equation*}
  	and
  	\begin{equation*}
  		\begin{aligned}
  			t_{1}(s)\geq
  			\begin{cases}
  				t_{1}(9)=456-544\alpha>0,&\text{if $\alpha\in[0,\frac{1}{2}]$},\\
  				t_{1}(11)=752-864\alpha>0,&\text{if $\alpha\in(\frac{1}{2},\frac{5}{7}]$},\\
  				t_{1}(\frac{19-5\alpha}{5(1-\alpha)})=\frac{1}{25(1-\alpha)}(-400\alpha^2+740\alpha+584)>0,&\text{if $\alpha\in(\frac{5}{7},1)$}.\\
  			\end{cases}
  		\end{aligned}
  	\end{equation*}

  	\textbf{Subcase 2.1.2}. $s$ is even.
  		
  	Let $t_{2}(s)=9(1-\alpha)s^2-2(-7\alpha+13)s+8\alpha+8$. Then we obtain $t_{2}(8)=376-456\alpha$, $t_{3}(10)=648-752\alpha$ and $t_{2}(\frac{2(9-2\alpha)}{5(1-\alpha)})=\frac{1}{25(1-\alpha)}(-336\alpha^2+484\alpha+776)$.
  	
  	Since 
  	\begin{equation*}
  		\begin{aligned}
  			n=\frac{5}{2}s+1\geq
  			\begin{cases}
  				20,&\text{if $\alpha\in[0,\frac{1}{2}]$},\\
  				25,&\text{if $\alpha\in(\frac{1}{2},\frac{5}{7}]$},\\
  				\frac{7}{1-\alpha}+3,&\text{if $\alpha\in(\frac{5}{7},1)$},\\
  			\end{cases}
  		\end{aligned}
  	\end{equation*}
  	we have
  	\begin{equation*}
  		\begin{aligned}
  			\frac{2(-7\alpha+13)}{2(9-9\alpha)}<
  			\begin{cases}
  				3<8\leq s,&\text{if $\alpha\in[0,\frac{1}{2}]$},\\
  				4<10\leq s,&\text{if $\alpha\in(\frac{1}{2},\frac{5}{7}]$},\\
  				\frac{2(9-2\alpha)}{5(1-\alpha)}\leq s,&\text{if $\alpha\in(\frac{5}{7},1)$},\\
  			\end{cases}
  		\end{aligned}
  	\end{equation*}
  	and
  	\begin{equation*}
  		\begin{aligned}
  			t_{2}(s)\geq
  			\begin{cases}
  				t_{2}(8)=376-456\alpha>0,&\text{if $\alpha\in[0,\frac{1}{2}]$},\\
  				t_{2}(10)=648-752\alpha>0,&\text{if $\alpha\in(\frac{1}{2},\frac{5}{7}]$},\\
  				t_{2}(\frac{2(9-2\alpha)}{5(1-\alpha)})=\frac{1}{25(1-\alpha)}(-336\alpha^2+484\alpha+776)>0,&\text{if $\alpha\in(\frac{5}{7},1)$}.\\
  			\end{cases}
  		\end{aligned}
  	\end{equation*}

  	Therefore, by (\ref{e999}), (\ref{e18}), $t_{1}(s)>0$ and $t_{2}(s)>0$, we have $\rho_{\alpha}(G)<n-3$.

  	\textbf{Subcase 2.2}. $i=\lfloor\frac{3s}{2}\rfloor+2$.
  	
  	Obviously, $n=s+\lfloor\frac{3s}{2}\rfloor+2$, and the quotient matrix of $A_{\alpha}(G)$ in view of the partition $\{V((\lfloor\frac{3s}{2}\rfloor+2)K_{1}),V(K_{s})\}$ equals
  	\begin{equation*}
  		B_{3}=	\begin{pmatrix}
  			\alpha s & (1-\alpha)s\\
  			(1-\alpha)(\lfloor\frac{3s}{2}\rfloor+2) & \alpha n-\alpha s+s-1\\
  		\end{pmatrix}.
  	\end{equation*}
  	
  	Then the characteristic polynomial of $B_{3}$ is
  	\begin{equation*}
  		\begin{aligned}
  			f_{B_{3}}(x)&=x^2-(\alpha n+s-1)x+\alpha^2sn-(\alpha^2-2\alpha+1)s\lfloor\frac{3s}{2}\rfloor-(\alpha^2-\alpha)s^2-(2\alpha^2-3\alpha+2)s\\
  			&=x^2-(\alpha n+s-1)x+(2\alpha-1)s\lfloor\frac{3s}{2}\rfloor+\alpha s^2+(3\alpha-2)s,
  		\end{aligned}
  	\end{equation*}
  	and $\rho_{\alpha}(G)$ is the largest root of  $f_{B_{3}}(x)=0$ by Lemma \ref{d2}. By a simple computation, we have
  	\begin{equation}\label{e1111}
  		\begin{aligned}
  			\rho_{\alpha}(G)=\frac{\alpha n+s-1+\sqrt{(\alpha n+s-1)^2-4((2\alpha-1)s\lfloor\frac{3s}{2}\rfloor+\alpha s^2+3\alpha s-2s)}}{2}.
  		\end{aligned}
  	\end{equation}
  	
  	Now we show $\rho_{\alpha}(G)<n-3$. It follows from $n=s+\lfloor\frac{3s}{2}\rfloor+2$ that
  	\begin{equation}\label{e19}
  		\begin{aligned}
  			&\hspace{1.3em}(2(n-3)-(\alpha n+s-1))^2-(\alpha n+s-1)^2+4((2\alpha-1)s\lfloor\frac{3s}{2}\rfloor+\alpha s^2+3\alpha s-2s)\\
  			&=(4-4\alpha)n^2-(4s+20-12\alpha)n+(8\alpha-4)s\lfloor\frac{3s}{2}\rfloor+4\alpha s^2+4\alpha s+8s+24\\&=
  			\begin{cases}
  				9(1-\alpha)s^2-4(-2\alpha+5)s+9\alpha+3,&\text{if $s$ is odd},\\
  				9(1-\alpha)s^2-2(-\alpha+7)s+8\alpha,&\text{if $s$ is even}.\\
  			\end{cases}
  		\end{aligned}
  	\end{equation}
  		
  	\textbf{Subcase 2.2.1}. $s$ is odd.
  		
  	Let $t_{3}(s)=9(1-\alpha)s^2-4(-2\alpha+5)s+9\alpha+3$. Then we obtain $t_{3}(9)=552-648\alpha$, $t_{3}(11)=872-992\alpha$ and $t_{3}(\frac{17-3\alpha}{5(1-\alpha)})=\frac{1}{25(1-\alpha)}(-264\alpha^2+212\alpha+976)$.
  	
  	Since 
  	\begin{equation*}
  		\begin{aligned}
  			n=\frac{5}{2}s+\frac{3}{2}\geq
  			\begin{cases}
  				20,&\text{if $\alpha\in[0,\frac{1}{2}]$},\\
  				25,&\text{if $\alpha\in(\frac{1}{2},\frac{5}{7}]$},\\
  				\frac{7}{1-\alpha}+3,&\text{if $\alpha\in(\frac{5}{7},1)$},\\
  			\end{cases}
  		\end{aligned}
  	\end{equation*}
  	we have
  	\begin{equation*}
  		\begin{aligned}
  			\frac{4(-2\alpha+5)}{2(9-9\alpha)}<
  			\begin{cases}
  				2<9\leq s,&\text{if $\alpha\in[0,\frac{1}{2}]$},\\
  				3<11\leq s,&\text{if $\alpha\in(\frac{1}{2},\frac{5}{7}]$},\\
  				\frac{17-3\alpha}{5(1-\alpha)}\leq s,&\text{if $\alpha\in(\frac{5}{7},1)$},\\
  			\end{cases}
  		\end{aligned}
  	\end{equation*}
  	and
  	\begin{equation*}
  		\begin{aligned}
  			t_{3}(s)\geq
  			\begin{cases}
  				t_{3}(9)=552-648\alpha>0,&\text{if $\alpha\in[0,\frac{1}{2}]$},\\
  				t_{3}(11)=872-992\alpha>0,&\text{if $\alpha\in(\frac{1}{2},\frac{5}{7}]$},\\
  				t_{3}(\frac{17-3\alpha}{5(1-\alpha)})=\frac{1}{25(1-\alpha)}(-264\alpha^2+212\alpha+976)>0,&\text{if $\alpha\in(\frac{5}{7},1)$}.\\
  			\end{cases}
  		\end{aligned}
  	\end{equation*}

  	\textbf{Subcase 2.2.2}. $s$ is even.
  		
  	Let $t_{4}(s)=9(1-\alpha)s^2-2(-\alpha+7)s+8\alpha$. Then we obtain $t_{4}(10)=464-552\alpha$, $t_{4}(10)=760-872\alpha$ and $t_{4}(\frac{2(8-\alpha)}{5(1-\alpha)})=\frac{1}{25(1-\alpha)}(-184\alpha^2-76\alpha+1184)$.
  	
  	Since 
  	\begin{equation*}
  		\begin{aligned}
  			n=\frac{5}{2}s+2\geq
  			\begin{cases}
  				20,&\text{if $\alpha\in[0,\frac{1}{2}]$},\\
  				25,&\text{if $\alpha\in(\frac{1}{2},\frac{5}{7}]$},\\
  				\frac{7}{1-\alpha}+3,&\text{if $\alpha\in(\frac{5}{7},1)$},\\
  			\end{cases}
  		\end{aligned}
  	\end{equation*}
  	we have
  	\begin{equation*}
  		\begin{aligned}
  			\frac{4(-2\alpha+5)}{2(9-9\alpha)}<
  			\begin{cases}
  				2<8\leq s,&\text{if $\alpha\in[0,\frac{1}{2}]$},\\
  				3<10\leq s,&\text{if $\alpha\in(\frac{1}{2},\frac{5}{7}]$},\\
  				\frac{2(8-\alpha)}{5(1-\alpha)}\leq s,&\text{if $\alpha\in(\frac{5}{7},1)$},\\
  			\end{cases}
  		\end{aligned}
  	\end{equation*}
  	and
  	\begin{equation*}
  		\begin{aligned}
  			t_{4}(s)\geq
  			\begin{cases}
  				t_{4}(8)=464-552\alpha>0,&\text{if $\alpha\in[0,\frac{1}{2}]$},\\
  				t_{4}(10)=760-872\alpha>0,&\text{if $\alpha\in(\frac{1}{2},\frac{5}{7}]$},\\
  				t_{4}(\frac{2(8-\alpha)}{5(1-\alpha)})=\frac{1}{25(1-\alpha)}(-184\alpha^2-76\alpha+1184)>0,&\text{if $\alpha\in(\frac{5}{7},1)$}.\\
  			\end{cases}
  		\end{aligned}
  	\end{equation*}
  		
  	Therefore, by (\ref{e1111}), (\ref{e19}), $t_{3}(s)>0$ and $t_{4}(s)>0$, we have $\rho_{\alpha}(G)<n-3$.
  	
  	Note that $\tau(n)>n-3$. Combining the above arguments, we conclude $\rho_{\alpha}(G)<n-3<\tau(n)$, which contradicts $\rho_{\alpha}(G)>\tau(n)$.
  	
  	By Case 1 and Case 2, we complete the proof of Theorem \ref{c2}.
  \end{proof}

  	\section{\textbf{Extremal graphs}}
  	
  	\hspace{1.5em}In this section, we claim that the condition in Theorem \ref{c2} is best possible.
  	
  	\begin{theorem}\label{c6}
  		Let $\alpha \in [0,1)$, $n$, $\tau(n)$ be as in Theorem \ref{c2}. Then $\rho_{\alpha}(K_{1}\vee(K_{n-3}\cup 2K_{1}))=\tau(n)$, and $K_{1}\vee(K_{n-3}\cup 2K_{1})$ contains no $\{P_{2},C_{3},P_{5},\mathcal{T}(3)\}$-factor.
  	\end{theorem}
  	
  	\begin{proof}
  		By the proof of Theorem \ref{c2}, we have $\rho_{\alpha}(K_{1}\vee(K_{n-3}\cup 2K_{1}))=\tau(n)$. Let $v$ be the vertex with the maximum degree of $K_{1}\vee(K_{n-3}\cup 2K_{1})$. Set $S=\{v\}$, then we infer $i(K_{1}\vee(K_{n-3}\cup 2K_{1})-S)=2>\frac{3}{2}|S|$. By Lemma \ref{d6}, the graph $K_{1}\vee(K_{n-3}\cup 2K_{1})$ contains no $\{P_{2},C_{3},P_{5},\mathcal{T}(3)\}$-factor. Therefore, the bound on $A_{\alpha}$-spectral radius established in Theorem \ref{c2} is sharp.
  	\end{proof}

  	\section*{\textbf{Funding}}
  	
  	\hspace{1.5em}This work is supported by the National Natural Science Foundation of China (Grant Nos. 12371347, 12271337).

\end{document}